\pdfoutput=1
\documentclass[a4paper,10pt]{amsart}
\usepackage[margin=20mm]{geometry}

\def\polhk#1{\setbox0=\hbox{#1}{\ooalign{\hidewidth
  \lower1.5ex\hbox{`}\hidewidth\crcr\unhbox0}}}
\usepackage[numbers]{natbib}
 \usepackage{graphics,color}
  \usepackage{graphicx}
  \usepackage{epsfig}
\usepackage{xypic}
\usepackage{amssymb,amsmath}
  \usepackage{amsthm}
\usepackage[colorlinks]{hyperref}
\newtheorem{lemma}{Lemma}
\newtheorem{definition}{Definition}
\newtheorem{theorem}{Theorem}
\newtheorem{corollary}{Corollary}
\theoremstyle{remark}
\newtheorem{example}{Example}
\newtheorem{remark}{Remark}

\newcommand{\St}{^{\textrm{st}}}
\newcommand{\Nd}{^{\textrm{nd}}}
\newcommand{\Rd}{^{\textrm{rd}}}
\newcommand{\Th}{^{\textrm{th}}}
\newcommand{\Span}[1]{\mathrm{Span}\left\{#1\right\}}
\newcommand{\graph}[1]{\mathrm{graph}\left(#1\right)}
\newcommand{\virg}[1]{``#1''}
\newcommand{\Gr}{\mathrm{Gr\,}}
\newcommand{\im}{\mathrm{im\,}}
\newcommand{\osc}{\mathrm{osc\,}}
\newcommand{\R}{\mathbb{R}}
\renewcommand{\P}{\mathbb{P}}

\newcommand{\C}{\mathcal{C}} 
\newcommand{\I}{\mathcal{I}} 
\newcommand{\SSigma}{\boldsymbol{\Sigma}}  

\newcommand{\E}{\mathcal{E}} 

\begin{document}




\title[Non--maximal integral elements of the Cartan plane]{Remarks on non--maximal integral elements of the Cartan plane in jet spaces}
\author[M. B\"achtold \and G. Moreno]%
{Michael B\"achtold* \and Giovanni Moreno**}

\newcommand{\acr}{\newline\indent}

\address{\llap{*\,}Lucerne University of Applied Sciences and Arts\acr
Technikumstrasse 21\acr
CH--6048 Horw, Switzerland}
\email{michael.baechtold@hslu.ch}

\address{\llap{**\,}Mathematical Institute in Opava\acr
Silesian University in Opava\acr
Na Rybnicku 626/1, 746 01 Opava, Czech Republic.}
\email{Giovanni.Moreno@math.slu.cz}

 \thanks{The second author is thankful to the   Grant Agency of the Czech Republic (GA \v CR)
for financial support under the project P201/12/G028, as well as to the organizers of the 2013 edition of the Workshop on Geometry of PDEs and Integrability (Teplice nad Be\v{c}vou, Czech Republic) for giving him the opportunity to discuss the problem hereby dealt with. The authors are indebted  to the MathOverflow forum for   providing crucial   hints and clues, and appreciated the meticulous examination of the manuscript carried out by  the anonymous referee.}

\begin{abstract}
There is a natural filtration on the    space  of degree--$k$ homogeneous polynomials  in $n$ independent variables with coefficients in the algebra of smooth functions on the Grassmannian   $\Gr(n,s)$, determined by the tautological bundle.  
In this paper we show that the space of $s$--dimensional integral elements of a Cartan plane on $J^{k-1}(E,n)$, with $\dim E=n+m$,  has an affine bundle structure modeled by the  the so--obtained   bundles over $\Gr(n,s)$,   
and we study  a natural distribution associated with it. As an example, we show that a third--order nonlinear PDE of Monge--Amp\`ere type is not contact--equivalent to a quasi--linear one.
\end{abstract}

\subjclass[2010]{14M15, 35A99, 35A30, 35B99, 57R99, 58A20}

\maketitle

\tableofcontents

\section{Introduction}

Let $\Gr(n,s)$ be the Grassmannian manifold of   $s$--dimensional vector subspaces of $L:=\R^n$. The algebra 
\begin{equation}\label{eqAlgebra}
S^\bullet T^\ast=\C^{\infty}(\Gr(n,s))\otimes_\R S^\bullet L^\ast
\end{equation}
 of polynomials functions on $L$ with coefficients in the algebra $\C^{\infty}(\Gr(n,s))$, regarded as a $\C^{\infty}(\Gr(n,s))$--module, is nothing but the  module   of smooth sections of the full symmetric algebra of the dual of the trivial bundle
\begin{equation}\nonumber
\xymatrix{
T:=\Gr(n,s)\times L\ar[d]\\
\Gr(n,s).
}
\end{equation}
Here, as everywhere else in this paper, we took the liberty of using  the same symbol both for the total space and for the module of sections of a vector bundle. Moreover, since the bundle--theoretic features of   $S^\bullet T^\ast$ will be exploited only through  its homogeneous components, we keep calling a \virg{bundle} even  such  an infinite--rank module.    \par
Recall that the \emph{tautological bundle}  $\SSigma\subseteq T$, which   is defined by\footnote{By convention, we shall denote by $E_p$ the fiber at $p$ of a bundle $E$: this should clarify the curious  formula \eqref{eqTautologica} above.}
\begin{equation}\label{eqTautologica}
\SSigma_\Sigma=\Sigma\quad\forall\Sigma\in \Gr(n,s),
\end{equation}
fits into the so--called \emph{universal    sequence}
 \begin{equation}\label{eqSU}
\xymatrix{
\SSigma\ar[r]\ar[dr] ^{}& T\ar[r]\ar[d]^{} & \frac{T}{\SSigma}\ar[dl]\\
&\Gr(n,s).&
}
\end{equation}
Much as a submanifold $N\subseteq M$ determines a filtration\footnote{This is the so--called \virg{$\mu$--adic filtration} exploited  in the geometric theory of singularities of PDEs \cite{MR966202}.} of the smooth function algebra $C^\infty(M)$ by the powers of the ideal $\mu_N$ of the functions vanishing on $N$, the sub--bundle  \eqref{eqTautologica}  determines a  filtration of  the symmetric algebra \eqref{eqAlgebra}.  The key step is to introduce the sequence \eqref{eqSUdual} below, which can be regarded as a sort of  \virg{dual} of the     sequence of bundles \eqref{eqSU}: 
 \begin{equation}\label{eqSUdual}
\xymatrix{
S^\bullet\SSigma^\ast\ar[dr] ^{}&S^\bullet T^\ast\ar[l]\ar[d]^{} & \mu_{\SSigma}\ar[dl]^{}\ar[l]\\
&\Gr(n,s).&
}
\end{equation}
Again, we interpret each    homogeneous component of the modules appearing in  \eqref{eqSUdual}   as a vector bundle over $\Gr(n,s)$, and retain the name \virg{bundle} for the whole   modules. In particular, the sub--bundle $\mu_{\SSigma}$ determines   a natural filtration
\begin{equation}\nonumber
S^\bullet T^\ast\supseteq \mu_{\SSigma}\supseteq \mu_{\SSigma}^2\supseteq \cdots, 
\end{equation}
to which it corresponds a tower of vector bundles
\begin{equation}\nonumber
\xymatrix{
\frac{S^\bullet T^\ast}{\mu_{\SSigma}} \ar[dr] & \ar[l]\frac{S^\bullet T^\ast}{\mu^2_{\SSigma}}  \ar[d]& \ar[l]\cdots\ar[dl]\\
&\Gr(n,s).&
}
\end{equation}
Symbol   $J^k(E,n)$ will denote the space of $k$--jets of $n$--dimensional submanifolds of $E$, being  $E$      an arbitrary    $(n+m)$--dimensional manifold. 
In this paper we prove the following result (see also Figure \ref{Fig}). 
\begin{theorem}\label{ThEgregio}
The space $\I_s(\C)$ of  $s$--dimensional horizontal integral elements of the Cartan plane $\C$ on $J^{k-1}(E,n)$ at a    point $\theta\in J^{k-1}(E,n)$ is an affine bundle over $\Gr(n,s)$, modeled by the   $k\Th$ homogeneous component of $ \frac{S^\bullet T^\ast}{\mu^k_{\SSigma}}\otimes_\R N$, with $N\cong\R^m$.
\end{theorem}
The main motivation for Theorem \ref{ThEgregio}  comes from the theory of geometric singularities of solutions of nonlinear PDEs (see \cite{LucaChar} and references therein). Roughly speaking, the space $\I_n(\C)$ corresponds to the $k$--jets of regular solutions whose $k-1\St$ jet is $\theta$, where by \emph{regular} we mean projecting without degenerations to the lower--order jet spaces. On the other hand,  $\I_s(\C)$, with $s<n$, corresponds to the so--called \emph{singular} solutions (of \emph{type} $n-s$), i.e., projecting with degeneration.\footnote{Recall that \emph{degeneracy} here refers to the map, not the   solution itself, which is   a genuine smooth submanifold of $J^{k-1}(E,n)$}  In this perspective, Theorem \ref{ThEgregio} is a first step towards the description, in terms of Grassmannian manifolds and their natural structures,  of the geometry of $\I_s(\C)$, which is a somewhat less known object as compared to $\I_n(\C)$, i.e., the well--known fiber   $J^k(E,n)_\theta$. It is worth mentioning that  a parallel  investigation   has been  carried out  recently     by one of us (GM) in the context of jet spaces of infinite order \cite{MorenoCauchy}.\par
We stress that, except for the concluding   Section \ref{exa}, the point $\theta$ is fixed once and for all, so that the topology of $E$ plays no role whatsoever in our analysis. Moreover, since, as it will turn out, the space   $\I_s(\C)$ possesses    an affine bundle structure, it is not restrictive to fix also an $n$--dimensional horizontal integral element $L\leq \C$, which, in its turn, can be identified with $\R^n$. Similarly, if $\theta=[L_0]_y^{k-1}$, then      the \emph{normal space} $N$  at $\theta$, which is defined as $\frac{T_y E}{ T_y L_0}$, is completely determined by $\theta$ and as such it will be identified with $\R^m$. The only \virg{variable objects} we shall deal with are the integral elements of $\C$ and their \virg{shadows} on $L$, which are elements  of $\Gr(n,s)$.
\subsection{Notations and conventions}
By the symbol $L$ (resp., $N$) we shall always mean $\R^n$ (resp., $\R^m$), being $n$ and $m$ fixed values throughout the paper; accordingly, $\R[x^1,\ldots,x^n]$ can be written as $S^\bullet L^\ast$. 
 Concerning jet spaces and related notions (coordinates, total derivatives, etc.), our notations and conventions comply with the one used in the book \cite{MR1670044} or in the review paper \cite{MR2813504}. We managed not to use a  proprietary notation, except for  the Grassmannian   $\check{\Gr}(\C,s)$ of \emph{horizontal} $s$--dimensional subspaces of $\C$, reading $\check{\P}\C$ when $s=1$.
\begin{figure}
\caption{Maximal, i.e., $n$--dimensional, horizontal integral subspaces of $\C$ can be labeled by polynomials: $L_p$, $L_q$, etc. In turn, the entire family of $s$--dimensional subspaces of a fixed $L_p$, which coincides with $\Gr(n,s)$, is made of (non--maximal, horizontal)   integral elements. So, a pair $(\Sigma_0, p)$  made of a subspace of $L$ and a polynomial on $L$, can be used to identify a generic $s$--dimensional horizontal integral element of $\C$. Nevertheless, to make this identification one--to--one, some polynomials, namely those which lift $\Sigma_0$ to the same subspace of $\C$, need to be factored out, and the universal sequence over $\Gr(n,s)$ provides precisely these polynomials.
\label{Fig}}
 \epsfig{file=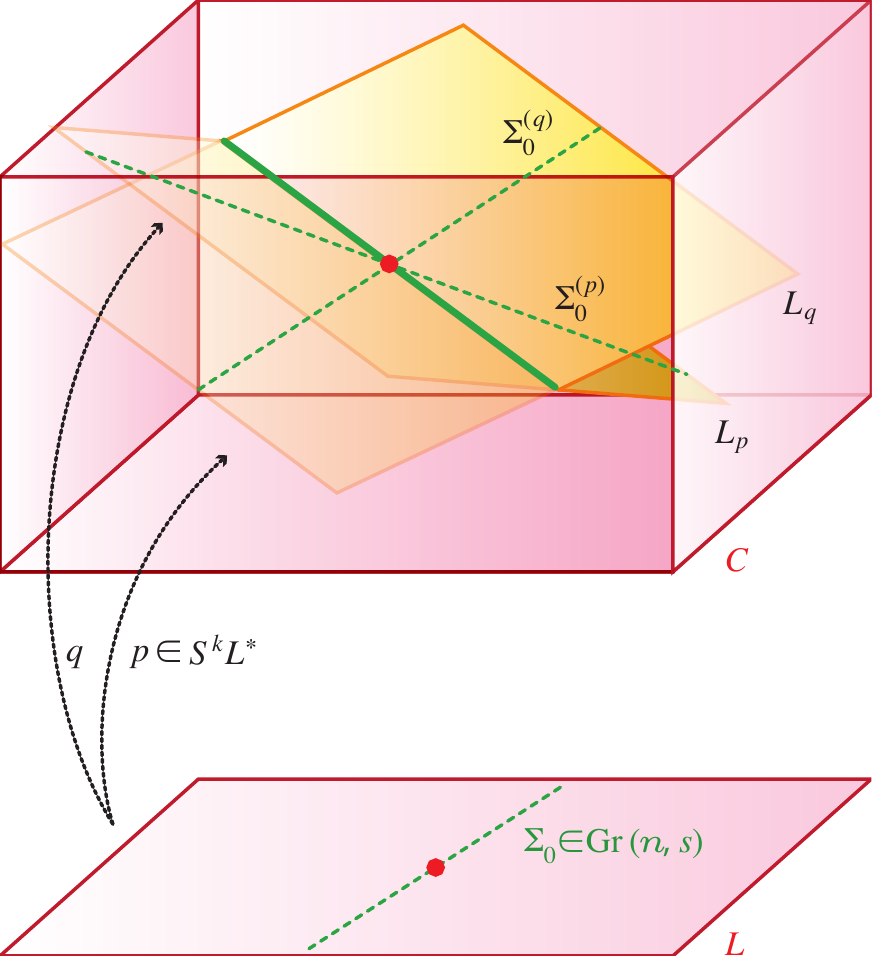}
\end{figure}
\subsection{Structure of the paper}
This paper is pivoted around Theorem \ref{ThEgregio}, proved in the central Section \ref{secPrTh}. In spite of the  plainness of its arguments, the general proof of Theorem \ref{ThEgregio} may look slightly  abstract, so that we deemed it convenient to dedicate the opening Section \ref{secMotExample} to a toy model, whose proof can be carried out in few elementary steps. In the rest of the paper we basically go over the same ideas and results sketched in Section \ref{secMotExample}, but  in a broader context and with more details. To this end, two gadgets need to be introduced: the so--called meta--symplectic structure on the Cartan plane (Section \ref{meta}), and a suitable incidence relation of isotropic elements, with respect to such a structure  (Section \ref{IsoFlag}).  In the  concluding   Section \ref{exa}  we obtain two easy but apparently original results which benefits from Theorem \ref{ThEgregio}, and speculate on its   envisaged range of applications.
\section{A motivating example}\label{secMotExample}
The first nontrivial example to test Theorem \ref{ThEgregio}  corresponds to the choice $n=3$. Let $\C=L\oplus L^\ast$, with $L=\Span{x_1,x_2,x_3}$ and $L^\ast=\Span{\xi^1,\xi^2,\xi^3}$, where $\langle x_i,\xi^j \rangle=\delta_i^j$, and denote by    
\begin{equation}\label{eqGraficoPoly}
\graph{p}:=\{l+p(l)\mid l\in L\}
\end{equation}
the graph of  $p=p_{ij}\xi^i\otimes \xi^j\in L^\ast\otimes L^\ast$. 
Call \emph{horizontal} a subspace $\Sigma\leq\C$ whose projection on $L$ is non degenerate, and define
\begin{equation}\label{eqDefGrHor}
\check{\Gr}(\C,s):=\{ \Sigma\in \Gr(\C,s)\mid \Sigma\textrm{ is horizontal}\}.
\end{equation}
Obviously, $\check{\Gr}(\C,s)$ is nonempty only for $s=0,1,2,3$. Observe that
\begin{equation}\label{eqGrH3}
\check{\Gr}(\C,3)=L^\ast\otimes_\R L^\ast
\end{equation}
since any horizontal 2D subspace of $\C$ is the graph of a homomorphism from $L$ to $L^\ast$. 
\subsection{The \virg{horizontal} Lagrangian Grassmannian of a 6D symplectic space}
Interestingly enough, the subspace $S^2L^\ast$ of symmetric forms on $L$ corresponds precisely to the homomorphisms whose graph is Lagrangian  with respect to the canonical symplectic structure
$
\omega=\xi^i\wedge x_i 
$, 
owing to the well--known fact that
\begin{equation}
\omega|_{\graph{p}}\equiv 0 \Leftrightarrow p\in S^2L^\ast.\label{eqGraficoIsotropo}
\end{equation}
In other words, we have proved that 
\begin{equation}\label{eqI3}
\I_3(\C)\cong S^2L^\ast, 
\end{equation}
i.e.,  the claim of Theorem \ref{ThEgregio} for $s=n=3$ and $m=1$. In such a case, $\Gr(3,3)$ reduces to a point, $\mu_{\SSigma}$ is the zero ideal, $N$ is a line and hence the $2\Nd$ homogeneous component of $ \frac{S^\bullet T^\ast}{\mu^2_{\SSigma}}\otimes_\R N$ is nothing but  $S^2L^\ast$. The so--obtained result just reformulates the fact that the generic fiber of $J^2(3,1)$ over $J^1(3,1)$ is the open and dense subset of the Lagrangian Grassmannian of the contact distribution, made of horizontal elements. 
\subsection{2D isotropic subspaces of a 6D symplectic space}\label{subEx2D6D}
We examine now   the case $s=2$, and try   to generalize formula \eqref{eqGrH3}. The key difference is that, in the case $s=3$, \emph{any} $\Sigma\in \check{\Gr}(\C,3)$ projects over the \emph{same} subspace $L$, whereas the projection $\underline{\Sigma}$ of a generic element $\Sigma\in \check{\Gr}(\C,2)$ ranges into the family of all 2D subspaces of $L$, i.e., it belongs to the Grassmannian $\Gr(3,2)$. In other words, there is a natural fibration
\begin{equation}\nonumber
\xymatrix{
\check{\Gr}(\C,2)\ar[d]^{}\\
\Gr(3,2),
}
\end{equation}
with
\begin{equation}\label{eqFibraGRC2}
\check{\Gr}(\C,2)_{\underline{\Sigma}}={\underline{\Sigma}}^\ast\otimes_\R L^\ast.
\end{equation}
Observe that formula \eqref{eqFibraGRC2} is \emph{not}, strictly speaking, the analog of \eqref{eqGrH3}, since it accounts for not the whole space $\check{\Gr}(\C,2)$, but just a fiber of it.
The only way to make it \virg{global}, is to use the universal sequence \eqref{eqSU}: indeed, the spaces $\Sigma$ and $L$  appearing in  \eqref{eqFibraGRC2} can be thought  of as the fibers of $\SSigma$ and $T$, respectively, over the same point $\Sigma$, so that formula 
\begin{equation}\label{eqFibraGRC2glob}
\check{\Gr}(\C,2) =\SSigma^\ast\otimes T^\ast
\end{equation}
actually   describes the space $\check{\Gr}(\C,2)$ as a tensor combination of natural vector bundles over $\Gr(3,2)$. It remains to describe the sub--bundle $\I_2(\C)$ made of isotropic 2D subspaces: since any such subspace is contained into a 3D (i.e., Lagrangian) subspace, in view of \eqref{eqGraficoPoly}, we can claim that for any $\Sigma\in \I_2(\C)$ there exists a homomprhism $p_{\underline{\Sigma}}\in S^2 L^\ast $ such that 
\begin{equation}\label{eqSigmaGEnerico}
\Sigma=\{\sigma+p_{\underline{\Sigma}}(\sigma)\mid \sigma\in\underline{\Sigma}\},
\end{equation}
i.e., $\Sigma$ is the graph of the restriction of $p_{\underline{\Sigma}}$ to $\underline{\Sigma}\leq L$. In other words, the space $S^2 L^\ast$ acts on the fiber $\I_2(\C)_{\underline{\Sigma}}$. The stabilizer of $\underline{\Sigma}$ is made by those $p\in S^2 L^\ast$ such that the subspace $\Sigma$  given in \eqref{eqSigmaGEnerico} coincides with $\underline{\Sigma}$, i.e., such that  $p|_{\Sigma}\equiv 0$. Being $p$ symmetric, the last condition is equivalent to the fact that $p\in \mu_\Sigma^2$, and hence it is proved that $\I_2(\C)_{\underline{\Sigma}}$ is modeled over the $2\Nd$ homogeneous component of the fiber of $ \frac{S^\bullet T^\ast}{\mu^2_{\SSigma}} $ over the   point $\underline{\Sigma}$.
 \subsection{An example of polar distribution}\label{subExPolar}
As an immediate consequence of \eqref{eqI3}, 
\begin{equation}\label{eqIsoFundS2}
T_L\I_3(\C)\cong S^2L^\ast.
\end{equation}
In the case of $\I_2(\C)$ there is no analog of isomorphism \eqref{eqIsoFundS2}, but just an inclusion:
\begin{equation}\label{eqTangenteI2}
T_\Sigma\I_2(\C)\leq T \check{\Gr}(\C,2) =\Sigma^\ast\otimes_\R \frac{\C}{\Sigma}.
\end{equation}
On the other hand,   $\Sigma$ is isotropic, and as such it is contained into the kernel of the map 
$
\C \stackrel{\omega}{\longrightarrow}\C^\ast\longrightarrow \Sigma^\ast
$, 
i.e., there is a homomorphism
$
\frac{\C}{\Sigma}\to \Sigma^\ast
$ 
which, combined with \eqref{eqTangenteI2}, yields
\begin{eqnarray}
T_\Sigma\I_2(\C) &\longrightarrow &\Sigma^\ast\otimes_\R\Sigma^\ast,\nonumber\\
p&\longmapsto & p^\sharp:= \omega(\,\cdot\, , p(\,\cdot\,)).\label{eqDefSharp}
\end{eqnarray}
It can be easily shown that $p^\sharp$ is symmetric, e.g., by a direct coordinate approach.  Indeed, $\graph{p}$ is isotropic if and only if there exists a $\widetilde{p}\in S^2L^\ast$ such that
\begin{equation}\nonumber
\graph{p}=\{\sigma+\widetilde{p}(\sigma)\mid \sigma\in  {\Sigma}\},
\end{equation}
i.e., $p=\widetilde{p}|_\Sigma$ modulo the natural projection  of $ \frac{\C}{\Sigma}$ onto $\frac{\C}{L}$:
\begin{equation}\label{diagrammaincasinato}
\widetilde{p}\in L^\ast\otimes L^\ast \longrightarrow \Sigma^\ast\otimes L^\ast =\Sigma^\ast\otimes \frac{\C}{L} \longleftarrow  \Sigma^\ast\otimes \frac{\C}{\Sigma}\ni p.
\end{equation}
Let $\Sigma=\Span{x_1,x_2}$ and $\frac{\C}{\Sigma}=\Span{x_3,L^\ast}$, so that 
\begin{equation}\label{eqTangVectCoord}
p=p_{\alpha j}\xi^\alpha\otimes \xi^j+p_\alpha \xi^\alpha\otimes x_3.
\end{equation}
Now, to apply the rightmost arrow of \eqref{diagrammaincasinato} to $p$ is the same as to set $p_\alpha=0$, and to impose that $p$ is the restriction of a symmetric $\widetilde{p}$ means to require that   that $p_{\alpha j}$ are symmetric for $j=1,2$.
In turn, this means   that 
$
p^\sharp=p_{\alpha \beta}\xi^\alpha\otimes \xi^\beta
$ 
is symmetric. We have just shown that the analog of isomorphism \eqref{eqIsoFundS2} for non--maximal isotropic elements is the epimorphism 
$
\sharp:T_\Sigma\I_2(\C)  \longrightarrow S^2\Sigma^\ast 
$. 
The \emph{polar plane}   $P_\Sigma\leq T_\Sigma\I_2(\C)$ is defined as 
\begin{equation}\label{defPolar}
P_\Sigma:=\ker\sharp,
\end{equation}
and $\Sigma\longmapsto P_\Sigma$ defines a distribution on $\I_2(\C)$ called \emph{polar}, firstly studied, to the authors best knowledge,\footnote{The name \emph{polar} is used here for the first time. The authors cannot be sure though that the same notion is not already present, in other guises, in the literature.} by one of us (MB) \cite{ThesisMichi}. The existence of this distribution was pointed out by A.M. Vinogradov \cite{PolarPrivate} during the  Ph.D thesis of the first author.\par
Imposing \eqref{defPolar} on the coordinate expression  \eqref{eqTangVectCoord} of $p$ we get
\begin{equation}\label{eqTangVectCoordPolar}
p=\xi^\alpha\otimes(p_{\alpha 3}  \xi^3+p_\alpha   x_3),
\end{equation}
i.e., $p$ takes its values in the subspace $\Span{x_3,\xi^3}$ of $\frac{\C}{\Sigma}$. The pre--image of $\im p$ under the natural projection $\C\to \frac{\C}{\Sigma}$ is a subspace of $\C$ canonically associated with $p$, known as its \emph{osculator}, and denoted by $\osc p$ \cite{MO98658}. So, \eqref{defPolar} is equivalent to the fact that $\osc p$  is contained in $L\oplus\Span{\xi^3}$, which is precisely $L\oplus\left( \frac{L}{\Sigma}\right)^\ast$, i.e., the $\omega$--orthogonal complement  $\Sigma^\perp$ of $\Sigma$ in $\C$.\par
Now we are able to reveal the geometric content of $P_\Sigma$. Let $\gamma_p(t)$ be an integral curve of $p\in T_\Sigma\I_2(\C) $, such that $\gamma_p(0)=\Sigma$. Then  $\{\gamma_p(t)\}_{-\epsilon<t<\epsilon}$ is a family of planes containing $\Sigma$, and we denote by $O_{\gamma_p}$ its linear envelope, i.e., the smallest linear subspace containing it: $\osc p$ is precisely the intersection of all the $O_{\gamma_p}$'s. Hence, $p\in P_\Sigma$ if and only if the deformation of $\Sigma$ determined by $p$ is, to first order, contained into $\Sigma^\perp$. 
\subsection{Concluding remarks}
In this preliminary section we proved main Theorem \ref{ThEgregio}, in the simplest case in which its statement does  not become trivial. We also pointed out the existence of a canonical distribution on $\I_s(\C)$, which possesses a transparent geometric interpretation.  A fuller structural account of this distribution will be the object of  a forthcoming study \cite{MichiForthcoming}. The remainder of this paper is dedicated to the proof of Theorem \ref{ThEgregio} and to some examples of the polar distribution. We quickly review the basic tools and notions needed for this.
\section{The meta--symplectic structure on $\C$}\label{meta}
We generalize the setting of Section \ref{secMotExample}   as follows. Recall that $\C$ was given as the direct sum $L \oplus L^\ast$: now, instead of $L^\ast$, i.e., linear functions on $L$, we take \emph{all} $N$--valued polynomials on $L$, of a certain degree $k-1$, i.e., the space $S^{k-1}L^\ast\otimes_\R N$. In other words, we set
\begin{equation}\label{eqDefC}
\C:=L\oplus (S^{k-1}L^\ast\otimes_\R N).
\end{equation}
Definition \eqref{eqDefGrHor} can be reproduced  verbatim in this broader context, in which \eqref{eqGrH3} becomes
\begin{equation}\label{eqGrH3gen}
\check{\Gr}(\C,n)=L^\ast\otimes_\R S^{k-1}L^\ast\otimes_\R N.
\end{equation}
The meta--symplectic form on $\C$ serves exactly the same purpose as the symplectic form in Section \ref{secMotExample} , i.e., it allows to single out the elements of $\check{\Gr}(\C,n)$ which correspond to the symmetric polynomials of degree $k$ on $L$. 
\begin{remark}
The \emph{polarization} $\hat{p}$ of a  polynomial $p$ of degree $k$  is $\frac{1}{k}$ times its differential $dp$, which can be understood as a linear map from $L$ to $S^{k-1}L^\ast$:
\begin{eqnarray}
L & \stackrel{\hat{p}}{\longrightarrow} & S^{k-1}L^\ast\otimes_\R N,\nonumber\\
l &\longmapsto & \hat{p}(l):=p(l,\underset{k-1\textrm{ symmetric entries}}{\underbrace{\, \cdot\,  \cdot\,  \cdot\,  \cdot\,   \cdot\,  \cdot\, \cdot\, \cdot\, \cdot\,}}).
\end{eqnarray}
\end{remark}
In other words, symmetric polynomials correspond to  {exact forms}, which in turn coincide with the closed ones. So, the sought--for meta--symplectic form $\Omega$ must tell the latter ones, much as the canonical symplectic form on the cotangent manifold allows to distinguish the graphs of closed forms.  Its definition is straightforward.\par
Just set
\begin{equation}\label{eqDefOmega}
\Omega^\prime(l,p):=\hat{p}(l)\in S^{k-2}L^\ast\otimes_\R N
\end{equation}
for any   $l\in L$ and $p\in S^{k-1}L^\ast\otimes N$.
\begin{lemma}
 There exists a unique $(S^{k-2}L^\ast\otimes N)$--valued 2--form $\Omega$ on $\C$ which extends the $\Omega^\prime$ defined by \eqref{eqDefOmega}, such that both $L$ and $S^{k-1}L^\ast\otimes_\R N$ are $\Omega$--isotropic.
\end{lemma}
\begin{proof}
 From \eqref{eqDefC} it follows that
 \begin{equation}
\Lambda^2(\C)=\Lambda^2(L)\oplus\left( L^\ast \otimes_\R\left(S^{k-1}L^\ast\otimes N \right)^\ast  \right)\oplus \Lambda^2(S^{k-1}L^\ast\otimes_\R N).
\end{equation}
Then    $\Omega:=0+\Omega^\prime+0$ is the desired form.
\end{proof}
\begin{definition}
 $\Omega$ is the \emph{meta--symplectic form} on $\C$.
\end{definition}
\begin{example}[Coordinates] If $L=\Span{x_1,\ldots,x_n}$ and $N=\Span{y_1,\ldots, y_m}$, then 
\begin{equation}\label{eqDefSkmuN}
S^{k-1}L^\ast\otimes N = \Span{ p^j_\sigma \xi^\sigma\otimes y_j\mid |\sigma|\leq k-1},
\end{equation}
where $\sigma=(\sigma_1,\ldots,\sigma_n)$ is a multi--index  and $\xi^\sigma$ is a short for $\xi^{\sigma_1}\xi^{\sigma_2}\cdots\xi^{\sigma_n}$.  Coefficients $p_\sigma^j$ appearing in \eqref{eqDefSkmuN}  can be taken as the basis elements  of the dual space
 \begin{equation}\label{eqDefSkmuNdual}
\left(S^{k-1}L^\ast\otimes N\right)^\ast = \Span{ p^j_\sigma  \mid |\sigma|\leq k-1}.
\end{equation}
 In these coordinates,
\begin{equation}\nonumber
\Omega=(\xi_i\wedge p_\sigma^j)\otimes (\xi^{\sigma-1_i}\otimes y_j).
\end{equation}
Here $\sigma-1_i:=(\sigma_1,\ldots,\sigma_{i-1},\sigma_i-1,\sigma_{i+1},\ldots,\sigma_n)$. 
\end{example}
\begin{lemma}
 There exists a unique $(S^{k-2}L^\ast\otimes N)$--valued 2--form on $\C$ satisfying \eqref{eqDefOmega}, and it is precisely the curvature form of the Cartan distribution of $J^{k-1}(E,n)$ at $\theta$.
\end{lemma}
\begin{example}
 When $k=2$ and $\dim N=1$, $\Omega$ is the symplectic form on the contact distribution of $J^1(n,1)$.
\end{example}
\section{Isotropic Grassmannians and flag manifolds}\label{IsoFlag}
We give now the central definition.
\begin{definition}
The subset
\begin{equation}\label{eqDefIl}
 \I_s(\C):=\{\Sigma\leq \C\mid \dim\Sigma=s,\ \Omega|_\Sigma\equiv 0,\ \Sigma\textrm{ horizontal} \}
\end{equation}
of the Grassmannian manifold of $s$--dimensional subspaces of $\C$ is called the \emph{isotropic Grassmannian}; the incidence relation
\begin{equation}\label{eqDefIln}
\I_{s,n}(\C):=\{(\Sigma,R)\mid \Sigma\leq R \}\subseteq  \I_s(\C)\times  \I_n(\C)
\end{equation}
is the \emph{isotropic (partial)  flag manifold}.\footnote{See \cite{MorenoCauchy} for more information on isotropic flags.}
\end{definition}
We stress that by \virg{horizontal} in  \eqref{eqDefIl} we mean that $\Sigma$ is transversal with respect to the fibers of $J^{k-1}\longrightarrow J^{k-2}$.
\begin{corollary}
$ \I_n(\C)$ is precisely the subset of   $\check{\Gr}(\C,n)$  (see \eqref{eqGrH3gen}) made of symmetric polynomials.
\end{corollary}
\begin{proof}
 Just go through  the steps \eqref{eqGraficoIsotropo}.
\end{proof}
Next result casts a bridge between the isotropic flag manifold and the universal sequence over $\Gr(n,s)$.
\begin{lemma}\label{lemTrivi}
 $\I_{s,n}(\C)$ can be canonically identified with $S^kT^\ast$.
\end{lemma}
\begin{proof}
It is enough to identify $\I_n(\C)$ with $S^kL^\ast$, and observe that $S^kT^\ast$ is the trivial bundle $\Gr(n,s)\times S^kL^\ast$. Indeed, a pair $(\Sigma_0, p)\in \Gr(n,s)\times S^kL^\ast$  can be identified with the integral flag $({\Sigma_0^{(p)}},L_p)$ where $L_p$ is the element of $\I_n(\C)$ which corresponds to $p\in S^kL^\ast$, and ${\Sigma_0^{(p)}}$ is the image of $\Sigma_0\subseteq\R^n\equiv L$ in the identification $L\equiv L_p$ (see also Figure \ref{Fig}).
\end{proof}
\begin{corollary}
  $\I_{s,n}(\C)$ is a smooth manifold and
  \begin{equation}
\dim \I_{s,n}(\C) = s(n-s)+ {n+k-1\choose k}.
\end{equation}
\end{corollary}
Observe that, from its bare definition \eqref{eqDefIln}, it is not so evident that $\I_{s,n}(\C)$ possesses the trivial affine bundle structure over $\Gr(n,s)$ revealed by above Lemma \ref{lemTrivi}. This insight on     $\I_{s,n}(\C)$ is crucial in order to understand the structure of $\I_{s}(\C)$. Indeed, there is a natural surjection 
$
\I_{s}(\C)\longrightarrow \Gr(n,s).
$
\begin{corollary}
 Diagram
 \begin{equation}\label{eqPrimaSequenza}
\xymatrix{\I_{s,n}(\C)\ar[rr]\ar[dr] && \I_{s}(\C)\ar[dl]\\ &\Gr(n,s)&}
\end{equation}
is commutative.
\end{corollary}
\section{Proof of Theorem \ref{ThEgregio}}\label{secPrTh}
The identification proved in Lemma \ref{lemTrivi},   may lead to suspect that the sequence  \eqref{eqPrimaSequenza} corresponds to the rightmost  triangle  from \eqref{eqSUdual}. As it turns out, this is correct if one replaces $\mu_{\SSigma}$ by its power $\mu^k_{\SSigma}$.\par
Before proving  Theorem \ref{ThEgregio}, it is convenient to explain it intuitively as follows.
\begin{remark}
  $\I_{s,n}(\C)$ is an affine bundle modeled over $S^kT^\ast$, while $\I_{s,n}(\C)\longrightarrow \I_{s}(\C)$ (see diagram \eqref{eqPrimaSequenza} above) is   modeled over $S^k\left(\frac{T}{\SSigma}\right)^\ast$, that is, over $\mu_{\SSigma}^k\cap S^k\SSigma^\ast$. Hence, it is natural to expect that $\I_{s}(\C)$ is modeled over $\frac{S^kT^\ast}{\mu_{\SSigma}^k\cap S^k\SSigma^\ast}$. 
\end{remark}
\begin{proof}[Proof of Theorem \ref{ThEgregio} ]
 Fix $\Sigma_0\in \Gr(n,s)$, and observe that 
 \begin{equation}\label{eqExJayMenoUno}
 \frac{S^k L^\ast}{S^k\left(\frac{L}{\Sigma_0}\right)^\ast}
\end{equation}
is  the   $k\Th$ homogeneous component of the fiber of    $ \frac{S^\bullet T^\ast}{\mu^k_{\SSigma}} $   over $\Sigma_0$.
On the other hand,
\begin{equation}\nonumber
\I_{s}(\C)_{\Sigma_0}=\{\Sigma\in \I_{s}(\C)\mid \underline{\Sigma}=\Sigma_0\},
\end{equation}
where $\underline{\Sigma}$ is the image of $\Sigma$ under the projection $\C\to L\cong\R^n$.\par
As in the proof of Lemma \ref{lemTrivi}, write down an element of  $\I_{s,n}(\C)_{\Sigma_0}$   as a pair
\begin{equation}\nonumber
({\Sigma_0^{(p)}},L_p),
\end{equation}
where, in the above notation,  $\underline{{\Sigma_0^{(p)}}}=\Sigma_0$. Observe also that  the bundle projection \eqref{eqPrimaSequenza} determines the map
\begin{eqnarray*}
\I_{s,n}(\C)_{\Sigma_0} &\longrightarrow & \I_{s}(\C)_{\Sigma_0},\\
({\Sigma_0^{(p)}},L_p) &\longmapsto & {\Sigma_0^{(p)}}.
\end{eqnarray*}
Now we can show that the  free and transitive action of $S^k L^\ast$ over $\I_{s,n}(\C)_{\Sigma_0}$ descends to a transitive but not free action over $ \I_{s}(\C)_{\Sigma_0}$ whose stabilizer is precisely $S^k\left(\frac{L}{\Sigma_0}\right)^\ast$, thus proving that \eqref{eqExJayMenoUno} is isomorphic to $\I_{s}(\C)_{\Sigma_0}$.\par
To this end, take $q\in S^k L^\ast$ and ${\Sigma_0^{(p)}}\in \I_{s}(\C)_{\Sigma_0}$, and act as follows:
\begin{equation}\label{eqCoolAction}
\Sigma_0^{(p)}\stackrel{q}{\longmapsto} \Sigma_0^{(p+q)}.
\end{equation}
Then the stabilizer at 0 is simply the subspace of polynomials $q$ such that $\Sigma_0^{(q)}=\Sigma_0$; but, by definition, $\Sigma_0^{(q)}\subseteq L_q$ is the graph of the linear map $q:L\longrightarrow S^{k-1}L^\ast$, restricted to $\Sigma_0$: hence, it coincides with $\Sigma_0$ if and only if the polarization of $q$ vanishes on $\Sigma_0$, i.e., if and only if $q\in S^k\left(\frac{L}{\Sigma_0}\right)^\ast$ (see also Figure \ref{Fig}).
\end{proof}
%
%
%
%
%
\begin{corollary}
 $\I_{s}(\C)$ is a smooth manifold and
 \begin{equation}\nonumber
\dim \I_{s}(\C) = l(n-l)+\left[{n+k-1\choose k} - {n-l + k -1 \choose k}\right]m.
\end{equation}
\end{corollary}
\subsection{The polar distributions }
We generalize now the notion of polar distribution given in Subsection \ref{subExPolar}.
\begin{corollary}
The linear map
\begin{eqnarray}
T_\Sigma\I_s(\C) &\stackrel{\sharp}{\longrightarrow }&S^k\Sigma^\ast\otimes_\R N,\nonumber\\
p&\longmapsto & p^\sharp:= \Omega(\,\cdot\, , p(\,\cdot\,)),\label{eqDefSharpGEN}
\end{eqnarray}
generalizing \eqref{eqDefSharp} is well--defined and surjective.
\end{corollary}
\begin{definition}
 $P_\Sigma:=\ker\sharp$ is the \emph{polar plane} at $\Sigma$, and $\Sigma\stackrel{P}{\longmapsto} P_\Sigma$ is the \emph{polar distribution} on $\I_s(\C)$.
\end{definition}
\begin{corollary}
 \begin{align}
\dim P&= s(n-s)+\nonumber  \\
&+\left[{n+k-1\choose k} - {n-s+ k -1 \choose k}- {s+k-1 \choose k}\right]m.\label{eqDimPi}
\end{align}
\end{corollary}
\section{Examples}\label{exa}
\subsection{Isotropic lines in $1\St$ order jets}
The case of one--dimensional isotropic subspaces of the Cartan plane on $J^1(E,n)$ is particularly simple, since all lines are isotropic. In this case, $\Gr(n,1)=\P^{n-1}$ and  $\I_1(\C)$ can be written as $\check{\P}\C$, the affine subspace of $\P\C$ made of horizontal lines. Hence, diagram \eqref{eqPrimaSequenza} reads
\begin{equation}\nonumber
\xymatrix{
\P^{n-1}\times (S^2L^\ast\otimes_\R N) \ar[rr]\ar[dr]&& \check{\P}\C\ar[dl]\\
& \P^{n-1}&
}
\end{equation}
and, by Theorem \ref{ThEgregio}, $\check{\P}\C$ is modeled by
\begin{equation}\label{FibraCasoSemplice}
\frac{S^2(T^\ast)}{S^2\left(\frac{T}{\Sigma}\right)^\ast}\otimes N.
\end{equation}
In particular, the rank of $\check{\P}\C$ is
$
m\cdot\dim\left(S^2(T^\ast)\right)-\dim\left(S^2\left(\frac{T}{\Sigma}\right)^\ast\right)=m\cdot n
$, 
so that  
$
\dim \check{\P}\C=n-1+mn=n-1+m+(n-1)m 
$, 
which is exactly the dimension of the first jet--prolongation of a rank--$m$ bundle $\eta$ over $\P^{n-1}$.\par
Observe that, if \eqref{FibraCasoSemplice} is split as 
$
\left(S^2\Sigma^\ast\oplus\left(\left(\frac{T}{\Sigma}\right)^\ast\otimes\Sigma^\ast\right)\right)\otimes N
$, 
then  the sub--bundle 
$
V_1:=\left(\frac{T}{\Sigma}\right)^\ast\otimes\Sigma^\ast
$, 
of rank $(n-1)\cdot m$ corresponds to the vertical distribution $J^1(\eta)\to J^0(\eta)$. On the other hand, the polar distribution $P_1$, whose dimension is $\dim V_1 +n -1$, corresponds to the Cartan distribution on $J^1(\eta)$ \cite{MichiForthcoming}. 
\begin{lemma}\label{lemmaEsercizio1}
 $\check{\P}\C$, with its polar distribution, is isomorphic to the $1\St$ jet prolongation $J^1(\eta)$, with its Cartan distribution, where $\eta$ is the rank--$m$ bundle $S^2\SSigma^\ast\otimes N$, being $\SSigma\to \P^{n-1}$ the tautological line bundle.
\end{lemma}
\begin{proof}
To begin with, choose local coordinates $(x^i,u^j,u_i^j)$ on $J^1(E,n)$, with $i\in\{1,\ldots, n\}$ and $j\in\{1,\ldots, m\}$. Then, locally, 
\begin{equation}\label{eqCoordinateCI}
\C=\Span{D_i,\partial_{u^j_i}},
\end{equation}
and 
$
\Omega=dx^i\wedge du_i^j\otimes\partial_{u^j} 
$.
Let $\Sigma=\Span{\ell}\in\P\C_\theta$, with
\begin{equation}\label{eqGeneratoreElle}
\ell=D_1|_\theta+b^\alpha D_\alpha|_\theta+f^j_i\left.\partial_{u_i^j}\right|_\theta,
\end{equation} 
where   $\alpha\in\{2,\ldots, n\}$. Then
\begin{equation}\label{eqCoordProjPiCi}
(b^2,\ldots,b^n, f^j_1,\ldots,f^j_n)
\end{equation}
 is a coordinate system for the open neighborhood $\Sigma^\ast\otimes \Sigma^c$  of $\Sigma$, in the    projective space $\P\C=\R\P ^{n(m+1)-1}$, where $
\Sigma^c:=\Span{D_\alpha|_\theta,\left.\partial_{u_i^j}\right|_\theta} 
$ 
is a complement of $\Sigma$. Since  $\Sigma^\ast\otimes \Sigma^c$    identifies with $T_\Sigma\P\C$ by means of the correspondence
\begin{equation}\nonumber
\left(h:\ell\mapsto B^\alpha D_\alpha|_\theta+F^j_i\left.\partial_{u^j_i}\right|_\theta\right)\leftrightarrow \xi_h=B^\alpha\partial_{b^\alpha}|_\theta+F^j_i\left.\partial_{f^j_i}\right|_\theta,\nonumber
\end{equation}
the fact that $\xi_h$ belongs to $P_\Sigma$ reflects on a condition on the 
  coefficients $B^\alpha$ and $F^j_i$. More precisely, 
\begin{eqnarray}
\xi_h\in P_\Sigma &\Leftrightarrow& \Omega_\theta(\ell,h(\ell))=0\nonumber\\
 &\Leftrightarrow& \Omega_\theta(D_1|_\theta+b^\alpha D_\alpha|_\theta+f^j_i\partial_{u_i^j}|_\theta,\nonumber\\  && B^{\alpha} D_{\alpha} |_\theta+F^j_i\partial_{u^j_i}|_\theta)=0\nonumber\\
  &\Leftrightarrow& F_i^j\Omega  \left.\left(D_1,\partial_{u^j_i}\right)\right|_\theta+b^\alpha F_i^j \Omega  \left.\left(D_\alpha, \partial_{u^j_i}\right)\right|_\theta +\nonumber\\ && + f_i^jB^{\alpha} \Omega\left.\left(\partial_{u^j_i},D_{\alpha} \right)\right|_\theta=0\nonumber\\
    &\Leftrightarrow& (F_1^j+b^\alpha F_\alpha^j-f_\alpha^j B^\alpha)\partial_{u^j}|_\theta=0\nonumber\\
    &\Leftrightarrow& F_1^j=f_\alpha^j B^\alpha-b^\alpha F_\alpha^j\quad \forall j=1,2,\ldots, m.\label{eqFinaliCampiExInvOrd1}
\end{eqnarray}
Equations \eqref{eqFinaliCampiExInvOrd1} shows that  the elements of  $P_\Sigma$ are in one--to--one correspondence with  the $(n-1)m$--tuples $(B^2,\ldots,B^n, F^j_2,\ldots,F^j_n)$. In particular,  we obtain a basis  
$\{
\xi_2,\ldots, \xi_n, \xi_2^j,\ldots, \xi_n^j 
\}$ 
of $P_\Sigma$   by choosing $F_\alpha^j=0$ and $B^\alpha=\delta^\alpha_\beta$ for the  $X_\beta$'s (with $\beta=2,3,\ldots,n$), and then choosing $B_\alpha=0$ and $F_\alpha^j=\delta_\alpha^\beta\delta_k^j$ for the $X^\beta_k$'s (with $\beta=2,3,\ldots,n$ and $k=1,2,\ldots, m$).
Notice that the values at $\Sigma$ of the vector fields
\begin{eqnarray}
X_\alpha &:=& \partial_{b^\alpha}+ f_\alpha^j \partial_{f_1^j},\label{eqCamCoordGlob1p}\\
X^\alpha_j &:=& \partial_{f_\alpha^j}- b^\alpha \partial_{f_1^j},\label{eqCamCoordGlob2p}
\end{eqnarray}
are, by construction, $\xi_\alpha$ and $\xi^\alpha_j $, respectively, i.e., $P$ is spanned by \eqref{eqCamCoordGlob1p} and \eqref{eqCamCoordGlob2p}. Moreover, 
\begin{eqnarray}
\vspace{3pt }[X_\alpha, X^\beta_j] &=& -2 \delta_\alpha^\beta\partial_{f_1^j}.\label{eqCamCoordGlob3p}
\end{eqnarray}
Observe that \eqref{eqCamCoordGlob1p}  {resembles} a total derivative, \eqref{eqCamCoordGlob2p}  is  {almost} a vertical vector field, and   \eqref{eqCamCoordGlob3p}  {is reminiscent of} the well--known commutation relation of the standard non--holonomic frame of a first order jet space. The idea is to look for a coordinate system which \virg{turns right} formulas    \eqref{eqCamCoordGlob1p}, \eqref{eqCamCoordGlob2p} and \eqref{eqCamCoordGlob3p}.  To this end, let
\begin{equation}\label{nuovecoordinateproiettive}
(y^2,\ldots, y^n, v^1,\ldots, v^m, v_2^j,\ldots,v_n^j)
\end{equation}
be a new coordinates  on  $\P\C$, given   by
\begin{eqnarray}
b^\alpha &:=& y^\alpha,\label{eqCamCoordGlob1}\\
f_\alpha^j &:=& v_\alpha^j,\label{eqCamCoordGlob2}\\
f_1^j &:=&2v^j-v^j_\alpha y^\alpha.\label{eqCamCoordGlob3}
\end{eqnarray}
Then it is immediate to see that
$
\partial_{y^\alpha}+v^j_\alpha \partial_{v^j}= X_\alpha$, 
$\partial_{v_\alpha^j}= X_j^\alpha$,  and 
 $-\delta_\alpha^\beta \partial_{v^j} = \vspace{3pt }[X_\alpha,X^\beta_j]
$.
In other words, \eqref{eqCamCoordGlob1}, \eqref{eqCamCoordGlob2} and \eqref{eqCamCoordGlob3} define a local diffeomorphism between $\P\C$ and $J^1(\R^{n-1+m},n-1)$, which sends $P$ into the Cartan distribution of $J^1(\R^{n-1+m},n-1)$.
\end{proof}
\begin{remark}
The previous proof in coordinates is an adaptation of the proof found in \cite{ThesisMichi}. A generalization of this result states that the polar distribution is a prolongation of a certain natural system of PDEs on a type of natural bundles. A precise statement with a coordinate invariant proof will appear in \cite{MichiForthcoming}.
\end{remark}
\subsection{Applications to nonlinear PDEs}
The polar distribution provides an important insight on the structure of non--maximal integral elements of $\C$ which can be effectively used to study nonlinear PDEs: indeed, all the proposed constructions stem from the contact structure on jet spaces and, as such, constitute a valuable source of invariants. In a sense, the whole geometric theory of PDEs can be seen as a chapter of contact geometry,  where the meta--symplectic structure represents  a sort of higher--order analog of the    symplectic structure on contact planes \cite{MR2352610}. The importance of non--maximal isotropic elements is that among them there are the  singular solutions of a given PDE, which, in some cases,\footnote{Multidimensional Monge--Amp\`ere equations are a remarkable and well--known example of such PDEs \cite{MR2985508}; a somewhat less familiar though formally analog case has been studied by one of us (GM) in the context of $3\Rd$ oder scalar PDEs \cite{3OrdMAEs} in two independent variables.} allow to reconstruct the equation itself. More precisely, to any $k\Th$ order PDE $\E\subseteq J^k(E,n)$ one can associate\footnote{Beware that now  the symbol $\C$ denotes   the whole Cartan distribution, not an its single plane.} its $s$--type singularity equation (originally appearing in \cite{MR923359})
\begin{equation}\nonumber
\Sigma_{[s]}\E\subseteq \I_{n-s}(\C),
\end{equation}
which is, by definition, a contact invariant of $\E$.\footnote{A pleasant review of the \virg{fold--type} case, i.e., when $s=1$, can be found in \cite{LucaChar}; additional  information about the more general cases can be retrieved from the references therein.} Furthermore, $\Sigma_{[s]}$ can be equipped with the restricted polar distribution, and this structure can be used to prove non--equivalence results.
\begin{example}[A non--equivalence result]\label{exFinale}
 The $3\Rd$ order PDE $\E:=\{u_{xxy}u_{yyy}-u_{xyy}^2=0\}$ is not contact equivalent to a quasi--linear one.
\end{example}
\begin{proof}
 It will be accomplished as follows. First, we show that $\Sigma_{[1]}\E$ contains $\check{\P} D$, where $D$ is the three--dimensional sub distribution 
\begin{equation}\label{eqDefDi}
D=\Span{D_1,D_2,\partial_{u_{xx}}}\subseteq \C.
\end{equation}
 Second, we show that the polar distribution, restricted to $\P D$, has dimension at least one. This means that $\Sigma_{[1]}\E$ contains a nonempty subset where the polar distribution is, at least, one--dimensional. Then we pass to the singularity equation of a generic quasi--linear PDE, and we show that the polar distribution is zero--dimensional in \emph{all} the points. Hence, the two cannot be contact--equivalent.\par
 Before proving the above--listed three steps, let us recall that $\E\subseteq J^3(\R^3,2)$, and also that the point $\theta\in J^2(\R^3,2)$ is fixed. The 5D Cartan distribution, which is generated by the two total derivatives\footnote{Here $D_i$ is truncated to the coefficients of $2\Nd$ order, since it is a tangent vector determined by $\theta\in J^2$, for $i=1,2$.} $D_1$ and $D_2$, and by the vertical vector fields $\partial_{u_{xx}},\partial_{u_{xy}}$ and $\partial_{u_{yy}}$, is denoted by $\C$. Hence, the singularity equation $\Sigma_{[1]}\E$ sits in $\I_1(\C)=\check{\P}\C$.\par
 In preparation for the first step, it is convenient to identify characteristic covectors with their (one--dimensional) kernels, henceforth called \emph{characteristic directions}. Then, recall\footnote{See, e.g, \cite{LucaChar}.} that the singularity equation $\Sigma_{[1]}\E$ is made of the projections of the characteristic directions for $\E$: in other words, a line $\Sigma\in \check{\P}\C$ belongs to $\Sigma_{[1]}\E$ if and only if there is a point $\widetilde{\theta}\in \E_\theta$ such that the corresponding integral plane $L_{\widetilde{\theta}}$ contains $\Sigma$, and $\Sigma$ is a characteristic direction for $\E$ at $\widetilde{\theta}$. Take now a generic line  
$
\Sigma=\Span{D_1+aD_2+c\partial_{u_{xx}}}\in \check{\P}D
$, 
where $a,c$ are projective coordinates, and observe that the point 
$
\widetilde{\theta}:=(c,0,0,0)\in J^3(\R^3,2)_\theta
$ 
 belongs, in fact, to $ \E_\theta$. Since
 \begin{equation}\label{eqElleThetaTilde}
L_{\widetilde{\theta}}=\Span{D_1+c\partial_{u_{xx}}, D_2},
\end{equation}
it is obvious that $\Sigma\subseteq L_{\widetilde{\theta}}$. Moreover, the symbol 
$
u_{yyy}(dx)^2dy-2u_{xyy}dx(dy)^2+u_{xxy}(dy)^3
$
 of $\E$   vanishes at $\widetilde{\theta}$, i.e., all directions lying in  $L_{\widetilde{\theta}}$ and, in particular, $\Sigma$,  are characteristic ones. We conclude that   
$
\Sigma_{[1]}\E\supseteq \check{\P}D
$, 
as desired.\par
The second step is almost self--evident. Indeed, by comparing \eqref{eqElleThetaTilde} and \eqref{eqDefDi}, one realizes  not only that any $\Sigma\in \check{\P}D$ is contained into the singularity equation $\Sigma_{[1]}\E$, but also  that this very $\Sigma$ sits into an integral plane $  L_{\widetilde{\theta}}$ which is \emph{entirely contained} into $D$. Hence, $\Sigma$ belongs to a one--parametric family $\{\Sigma_t\}\subseteq \Sigma_{[1]}\E$   each member of which lies into the same $  L_{\widetilde{\theta}}$. But $\{\Sigma_t\}$ is an integral curve of the polar distribution, and it develops within $ \Sigma_{[1]}\E$: hence,
\begin{equation}\label{eqDimIntersezMagUno}
\dim P|_{\Sigma_{[1]}\E}\geq 1,
\end{equation}
as promised.\par
Turn now to the third and last step and, towards an absurd, suppose   that $\E$ is quasi--linear: hence, its symbol is constant along the fiber $\E_\theta$. This fact reflects on the structure of the singularity equation $\Sigma_{[1]}\E$. Indeed, since   \emph{all} integral planes $L_{\widetilde{\theta}}$, with $\widetilde{\theta}\in \E_\theta$, are mutually and canonically identified each other, then it makes sense to claim that the intersection
\begin{equation}\label{eqIntersezioneSingPElle}
\Sigma_{[1]}\E \cap \P L_{\widetilde{\theta}}=\{ \Sigma=\Span{\ell}\mid q(\ell)=0 \}
\end{equation}
is made of the \emph{same} three elements (possibly counted with multiplicity), which are the (real) roots of a fixed, i.e., not depending on $\widetilde{\theta}$, $3\Rd$ order homogeneous polynomial $q$, which is the (in this sense, constant)   symbol of $\E$. The same reasoning as above shows that any integral curve of $P$, passing through $\Sigma\in \Sigma_{[1]}\E$ is tangent to $\P L_{\widetilde{\theta}}$, but the intersection of the latter with $\Sigma_{[1]}\E$ is 0--dimensional in view of  \eqref{eqIntersezioneSingPElle}, and hence
$
\dim P|_{\Sigma_{[1]}\E}=0$,  
contradicting \eqref{eqDimIntersezMagUno}.
\end{proof}
\subsection{Concluding remarks and perspectives}
The celebrated 1995 proof, due Bryant \& Griffiths, of an old conjecture of Sophus Lie shows that, under mild restrictions, any parabolic Monge--Amp\`ere equation is contact--equivalent to a quasi--linear one 
\cite{MR1334205}.  Example \ref{exFinale} above may be thought of as a negative counter--example of an analogous property, but formulated in the context of $3\Rd$ order Monge--Amp\`ere equations, which is an area little explored so far (see \cite{3OrdMAEs} and references therein).\par
In spite of its simplified settings (low order PDEs, two independent variables, and $s=1$), Example \ref{exFinale} should help to realize what is the applicative range   of the proposed geometric framework for the non--maximal integral  elements $\C$.
%
%
%
%
%
%

\end{document}